\documentclass[10pt]{amsart}
\usepackage{amsfonts,amssymb,yhmath}
\usepackage[all]{xy}

\newtheorem{thm}{Theorem}[section] 
\newtheorem{lem}[thm]{Lemma} 
\newtheorem{prop}[thm]{Proposition}
\newtheorem{cor}[thm]{Corollary}

\newtheorem*{thmA}{Theorem A}
\newtheorem*{thmB}{Theorem B}
\newtheorem*{thmC}{Theorem C}

\theoremstyle{remark}
\newtheorem{remark}[thm]{Remark}

\def\qq{\mathbb{Q}}
\def\pp{\mathbb{P}}
\def\rr{\mathbb{R}}
\def\zz{\mathbb{Z}}

\def\xx{\mathfrak{X}}
\def\ddc{\, dd^c \, }

\makeatletter
\@namedef{subjclassname@2010}{%
  \textup{2010} Mathematics Subject Classification}
\makeatother

\numberwithin{equation}{section}

\begin{document}

\title[Local heights on Galois covers of the projective line]{Local heights on Galois covers of the projective line}

\author{Robin de Jong}


\begin{abstract} Let $X$ be a smooth projective curve of positive genus defined
over a number field $K$. Assume given a Galois covering map $x \colon X \to \pp^1_K$ and a place $v$ of $K$. We introduce a local canonical height on
$X(\overline{K}_v)$ associated to $x$ as an integral with logarithmic integrand, generalizing Tate's local N\'eron function on an elliptic curve. The resulting global height can be viewed as a `Mahler measure' associated to $x$. We prove that the local canonical height can be obtained by averaging, and taking a limit, over divisors of higher order Weierstrass points on $X$. This generalizes previous results by Everest-n\'i Fhlathu\'in and Szpiro-Tucker. Our construction of the local canonical height is an application of potential theory on Berkovich curves in the presence of a canonical measure.
\end{abstract}

\maketitle
\thispagestyle{empty}

\section{Introduction}

Let $K$ be a number field. Let $X$ be a smooth projective curve of positive
genus $g$ defined over $K$, endowed with a Galois covering map $x \colon X \to \pp^1_K$ of degree $N$ to the projective line over $K$. Assume that $x$ is totally ramified at some point $o \in X(K)$. The basic example that we have in mind is that of a hyperelliptic curve $X$ endowed with a hyperelliptic map. In this paper we introduce and study, for each place $v$ of $K$, a local height function:
\[ \lambda_v \colon X(\overline{K}_v) \setminus \{o\} \to \rr  \]
associated to $x$, generalizing the well known N\'eron-Tate local height function on an elliptic curve over $K$. Here $K_v$ denotes the completion of $K$ at $v$, and $\overline{K}_v$ an algebraic closure of $K_v$.

Fix a $v$-adic absolute value $|\cdot|_v$ on $K_v$, and assume that a coordinate function has been chosen on $\pp^1_K$ such that $o$ maps to $\infty$ under $x$. Our local height function will be obtained as an integral with a logarithmic integrand:
\[ \lambda_v(p) = \frac{1}{N} \int_{X_v} \log |x-x(p)|_v \, \mu_v \,  , \]
where $X_v$ is a canonical analytic space associated to $X$ at $v$ and $\mu_v$ a
canonical probability measure on $X_v$. To be precise, $X_v$ will be the Berkovich analytic
space associated to $X \otimes \hat{\overline{K}}_v$ if $v$ is non-archimedean,
and $X_v$ will be the complex analytic space $X(\overline{K}_v)$ if $v$ is
archimedean. The measure $\mu_v$ will be the canonical Arakelov measure on
$X_v$, to be defined in Section~\ref{potential} below.
\begin{thmA} The function $ \lambda_v \colon X(\overline{K}_v) \setminus \{o\} \to \rr $ is a local Weil height with respect to the divisor $o$ on $X$. The difference $\lambda_v(p) - \frac{1}{N} \log |x(p)|_v$ tends to zero as $p$ tends to $o$ on $X(\overline{K}_v)$.
\end{thmA}
For $p \in X(K) \setminus \{o\} $ we will see that $\lambda_v(p)$ vanishes for almost all $v$. Let $\overline{K}$ be an algebraic closure of $K$, and let $M_K$ be the set of places of $K$. The function $h_x \colon X(K) \setminus \{o\} \to \rr$ defined via:
\[ h_x(p) = \frac{1}{[K \colon \qq]} \sum_{v \in M_K} n_v \, \lambda_v(p) =
\frac{1}{[K \colon \qq]} \frac{1}{N} \sum_{v \in M_K} n_v \int_{X_v} \log |x-x(p)|_v \, \mu_v \, , \]
where the $n_v$ are suitable local factors, extends to a global Weil height with respect to $o$ on $X(\overline{K})$. We note that $[K \colon \qq]$ times this height can be viewed as a `Mahler measure' associated to the map $x \colon X \to \pp^1_K$.

It will be no surprise that $h_x$ relates directly to the N\'eron-Tate height on the jacobian of $X$ via the Abel-Jacobi map based at $o$ (cf. Proposition \ref{NT}). Because of this connection, and the explicit nature of our local heights, the results in this paper may have an application to the actual calculation of N\'eron-Tate heights of rational points on jacobians. One possible starting point is the result in Theorem~C below. Another possible starting point is to try to develop a `formulaire', \`a la Tate's formulaire for elliptic curves (cf. \cite{sil2}, Chapter~VI.4), for $\lambda_v(p)$ based, if $v$ is non-archimedean, on the type of the reduction graph of $X$ at $v$, and the specialisation of $p$ onto that graph.

The function $\lambda_v$ is integrable against $\mu_v$. This leads to a local invariant:
\begin{equation} \label{localinvariant}
\int_{X_v} \lambda_v \, \mu_v = \frac{1}{N} \int_{X_v} \int_{X_v} \log |x(p)-x(q)|_v \mu_v(p) \mu_v(q) \end{equation}
associated to each $v$. This invariant vanishes for almost all $v$ and hence the global invariant:
\begin{equation} \label{globalinvariant} \sum_{v \in M_K} n_v \int_{X_v} \lambda_v \, \mu_v = \frac{1}{N} \sum_{v \in M_K} n_v \int_{X_v} \int_{X_v} \log |x(p)-x(q)|_v \mu_v(p) \mu_v(q)
\end{equation}
is well-defined.
One expects this invariant to be comparable with more `classical' invariants of $X$, such as the (admissible) self-intersection of the relative dualising sheaf of $X$ as studied in \cite{zh}. We prove for $x$ a hyperelliptic map that this is indeed the case:
\begin{thmB} Assume $X$ is a hyperelliptic curve of genus $g \geq 2$ and assume that $x \colon X \to \pp^1_K$ is a hyperelliptic map. Let $(\omega,\omega)_a$ be the admissible self-intersection of the relative dualising sheaf of $X$. Then the formula:
\[ \sum_{v \in M_K} n_v \int_{X_v} \lambda_v \, \mu_v = \frac{1}{2} \sum_{v \in M_K} n_v \int_{X_v} \int_{X_v} \log |x(p)-x(q)|_v \mu_v(p) \mu_v(q) = \frac{ (\omega,\omega)_a}{4g(g-1)} \]
holds.
\end{thmB}
In fact we prove that in general the invariant (\ref{globalinvariant}) equals $-(o,o)_a$, where $(o,o)_a$ is the admissible self-intersection of the point $o$. For hyperelliptic curves, or more generally curves such that $(2g-2)o$ is a canonical divisor, this specializes to the formula in Theorem B. For hyperelliptic curves we are able to analyze the local invariants (\ref{localinvariant}) in more detail and connect them to a local invariant that we introduced in an earlier paper \cite{dj}.

Our next result states that the local heights $\lambda_v(p)$ for $p \in X(K) \setminus \{o\}$ can be obtained by averaging, and taking a limit, over higher order Weierstrass points on~$X$. More precisely, let $X_n$ for $n \geq g$ be the divisor of Weierstrass points of the line bundle $\mathcal{O}_X(o)^{\otimes n+g-1}$ on $X$ (see Section \ref{theoC} for definitions). This $X_n$ is an effective divisor of degree $gn^2$, generalizing the divisor of $n$-torsion points on $(X,o)$ when $(X,o)$ is an elliptic curve. We usually view $X_n$ as a $\mathrm{Gal}(\overline{K}/K)$-invariant multi-set in $X(\overline{K})$.
\begin{thmC} Let $p$ be a point in $X(K) \setminus \{o\}$. Assume that $p$ is
not in $X_n$ for at least one $n \geq g$. Let $v$ be a place of $K$. Then:
\[
 \frac{1}{gn^2} \sum_{q \in X_n \atop x(q) \neq x(p),\infty} \log|x(p)-x(q)|_v
\longrightarrow \int_{X_v} \log|x - x(p)|_v \, \mu_v
\]
for some sequence of natural numbers $n$ tending to infinity. In the sum on the left hand side, the points in $X_n$ are counted with multiplicity.
\end{thmC}
It is clear that there can be at most finitely many points $p$ in $X(\overline{K})$ such that $p \in X_n$ for all $n \geq g$. In any case, apart from some `obvious' cases, such points appear to be extremely rare \cite{sv}. If $p \notin X_n$ for some $n \geq g$ then $p \notin X_n$ for infinitely many $n \geq g$. The sequence of these $n$ can be taken in Theorem C. It seems likely that Theorem C holds without the assumption that $p$ should not be in every $X_n$, and that the limit can be taken over all natural numbers.

A celebrated theorem of Mumford-Neeman \cite{ne} implies that, at least if $v$ is archimedean, the multi-sets $X_n$ are weakly equidistributed with respect to $\mu_v$. As far as we know, the equidistribution of Weierstrass points at non-archimedean $v$ has not yet been proven. In some sense the theorem of Mumford-Neeman `explains' Theorem C but it seems there is no direct implication. For example, the condition that the point $p$ should be algebraic seems to be essential.

In the case that $(X,o)$ is an elliptic curve, Theorem C occurs in an article by Everest and n\'i Fhlathu\'in \cite{enf}. A result similar in spirit to Theorem C but in the context of dynamical systems on $\pp^1_K$ has been proven by Szpiro and Tucker \cite{st2}. They prove that if $\varphi \colon \pp^1_K \to \pp^1_K$ is a non-constant rational map of degree $d>1$, and $p \in \pp^1(K) \setminus \{\infty \}$ is a rational point, then:
\[  \frac{1}{d^n} \sum_{q \in \mathrm{Fix}_n \atop q \neq p,\infty} \log|x(p)-x(q)|_v
\longrightarrow \int_{\pp^1_v} \log|x - x(p)|_v \, \mu_{\varphi,v} \]
as $n \to \infty$. Here $\mathrm{Fix}_n$ is the multi-set (divisor) of fixed points of the iterate $\varphi^n$ for all positive integers $n$, and $\mu_{\varphi,v}$ is the local canonical measure on $\pp^1_v$ associated to $\varphi$ by Call-Silverman \cite{cs}. The results of Everest-n\'i Fhlathu\'in and Szpiro-Tucker were proved using diophantine approximation in some form. We will prove Theorem C as a corollary of Faltings's diophantine approximation results on abelian varieties~\cite{fa}.

A possible application of Theorem C is a method to actually calculate, or at least approximate, the canonical height of $p$ by evaluating a sequence of polynomials--the `division polynomials' associated to the $X_n$--at $x(p)$. In the case of elliptic curves, this method is discussed in \cite{ewNY}. The next interesting case would be the case of genus two (or more generally, hyperelliptic) curves. The division polynomials associated to the $X_n$ are then known to satisfy some explicit recurrence relations that one could use \cite{ca}. The upshot of this method is that one does not need algebraic equations for the jacobian of $X$. We have stated our result as Theorem~\ref{EW} below.

Let us summarize the contents of this paper. In Section \ref{potential} we recall some useful facts from potential theory on Berkovich spaces associated to curves, mainly based on Thuillier's thesis \cite{th}. This is the `right' context to make sense of the integrals involved in defining $\lambda_v$ in the non-archimedean setting. We also connect the theory with Zhang's theory of local admissible pairing on a curve \cite{zh}.

In Section \ref{global} we extend the theory to the global setting and make the connection, following a classical result of Hriljac-Faltings, with the
N\'eron-Tate height on the jacobian.

From Section \ref{lambda} on we assume $X$ is endowed with a Galois covering map $x \colon X \to \pp^1_K$. We start by introducing $\lambda_v$ and proving Theorem A. We verify that for $(X,o)$ an elliptic curve, the function $\lambda_v$ coincides with Tate's classical N\'eron function. Finally we prove Theorem~B. 

We analyze the case of hyperelliptic curves more closely in Section \ref{hyperelliptic}. Theorem~C is proved in Section \ref{theoC}.

\section{Potential theory on Berkovich analytic curves}  \label{potential}

Let $X$ be a geometrically connected smooth projective curve over a local field
$(k,|\cdot|)$. Let $\overline{k}$ be an algebraic closure of
$k$, and let $\hat{\overline{k}}$ be the completion of $\overline{k}$.
The absolute value $|\cdot|$ extends in a unique way to $\hat{\overline{k}}$. One has
associated to $X$ a locally ringed space $\xx=(|\xx|,\mathcal{O}_\xx)$, where the
underlying topological space $|\xx|$ has the following properties: $|\xx|$ is
compact, metrisable, path-connected, and contains $X(\overline{k})$ with its topology
induced from $|\cdot|$ naturally as a dense subspace. If $k$ is archimedean, we
just take $X(\overline{k})$ itself, with its structure of complex analytic space; if
$k$ is non-archimedean we let $\xx$ be the Berkovich analytic space associated
to $X \otimes \hat{\overline{k}}$, as in \cite{be}.

Our purpose in this section is to put a
canonical probability measure $\mu_\xx$ on $|\xx|$, and to discuss a few results
from potential theory on $\xx$. Everything is standard for archimedean $k$; however for
non-archimedean $k$ the results seem to be less known. We base our discussion on the thesis of Thuillier \cite{th}, in particular Chapter~3. For more background on Berkovich spaces we refer to \cite{bak}. As an
application of the formalism we present a formula (see Proposition                   \ref{formulaintegral} below) for the integral
$\int_\xx \log |f| \, \mu_\xx$ where $f$ is an arbitrary non-zero rational function on $X\otimes \overline{k}$. This formula establishes a link with Zhang's theory of
local admissible pairing \cite{zh}.

We start by considering the natural exact sequence:
\[ 0 \longrightarrow \mathcal{H} \longrightarrow
\mathcal{A}^0 \, {\buildrel \ddc \over \longrightarrow} \,
\mathcal{A}^1  \longrightarrow 0 \]
of sheaves of $\rr$-vector spaces on $\xx$. Here $\mathcal{H}$ is the sheaf of
harmonic functions on $\xx$, $\mathcal{A}^0$ is the sheaf of smooth functions on
$\xx$, $\mathcal{A}^1$ is the sheaf of smooth forms on $\xx$, and $\ddc$ is the
Laplace operator. The sheaf $\mathcal{A}^0$ is in fact a sheaf of
$\rr$-algebras  and the sheaf $\mathcal{A}^1$ is naturally a sheaf of modules
over $\mathcal{A}^0$.

We let $\mathcal{A}^0(\xx)$ and $\mathcal{A}^1(\xx)$ be the spaces of global
sections of $\mathcal{A}^0$ and $\mathcal{A}^1$. Further we put
$\mathcal{D}^0(\xx) = \mathcal{A}^1(\xx)^*$ and $\mathcal{D}^1(\xx) =
\mathcal{A}^0(\xx)^*$ for their $\rr$-linear duals. We have a natural
$\rr$-linear integration map $\int_\xx \colon \mathcal{A}^1(\xx) \to \rr$ and a
natural $\rr$-bilinear pairing $\mathcal{A}^0(\xx) \times
\mathcal{A}^1(\xx) \to \rr$ given by
$(\varphi,\omega) \mapsto \int_\xx \varphi \, \omega$.
This pairing yields a natural commutative diagram:
\[ \xymatrix{ \mathcal{D}^0(\xx) \ar[r] & \mathcal{D}^1(\xx) \\
\mathcal{A}^0(\xx) \ar[u] \ar[r]^\ddc & \mathcal{A}^1(\xx)
\ar[u]    }  \]
where the upward arrows are injections, and
the map $\mathcal{D}^0(\xx)
\to \mathcal{D}^1(\xx)$ is the dual of the map
$\mathcal{A}^0(\xx)\, {\buildrel
\ddc \over \longrightarrow} \, \mathcal{A}^1(\xx)$,
also to be denoted by $dd^c$. The kernel of $\ddc
\colon \mathcal{D}^0(\xx) \to \mathcal{D}^1(\xx)$
is a one-dimensional $\rr$-vector space, to be identified with the set of constant
functions on $\xx$. The elements of $\mathcal{D}^\alpha(\xx)$ are called
$(\alpha,\alpha)$-currents; $(1,1)$-currents can be viewed as measures on
$|\xx|$. The unit element of $\mathcal{A}^0(\xx)$ gives, under the natural map
$\mathcal{A}^0(\xx) \to \mathcal{D}^1(\xx)^*$, a natural $\rr$-linear integration
map $\int_\xx \colon \mathcal{D}^1(\xx) \to \rr$, extending $\int_\xx$ on
$\mathcal{A}^1(\xx)$. Associated to each non-zero rational function $f$ on $X
\otimes \overline{k}$ we have a natural $(0,0)$-current $\log |f| \in
\mathcal{D}^0(\xx)$. For each closed point $p$ on $X\otimes \overline{k}$ we have a Dirac measure $\delta_p \in
\mathcal{D}^1(\xx)$, and by linear extension we obtain a measure $\delta_D \in
\mathcal{D}^1(\xx)$ for each divisor $D$ on $X\otimes \overline{k}$.
\begin{prop} \label{properties}
(i) (Poincar\'e-Lelong equation) Let $f$ be a non-zero
rational function on $X \otimes \overline{k}$. Then the equation: \[ \ddc \log |f| -
\delta_{\mathrm{div} f} = 0 \] holds in $\mathcal{D}^1(\xx)$. \\
(ii) (The Laplace operator is self-adjoint) We have: \[
\int_\xx \varphi \ddc \psi = \int_\xx \psi \ddc \varphi   \]
for all $\varphi, \psi \in \mathcal{D}^0(\xx)$. \\
(iii) (Existence
and uniqueness of Green's functions) Let $\mu \in \mathcal{A}^1(\xx)$ be a
smooth measure with $\int_\xx \mu=1$, and  let $p \in X(\overline{k})$. Then there
exists a unique current $g_{\mu,p} \in \mathcal{D}^0(\xx)$ such that: \[ \ddc g_{\mu,p} = \mu - \delta_p \, , \quad \int_\xx g_{\mu,p} \, \mu = 0 \, . \] The symmetry
relation $g_{\mu,p}(q) = g_{\mu,q}(p)$ holds for all $p \neq q \in X(\overline{k})$.
\end{prop}
\begin{proof} This is well-known for archimedean $k$. We find (i)--(iii)
respectively in \cite{th}, Proposition 3.3.15, Proposition 3.2.12 and
Proposition 3.3.13, for non-archimedean~$k$.
\end{proof}
Our next goal is to designate
a canonical probability measure $\mu_\xx \in \mathcal{A}^1(\xx)$. We assume from now on that the genus~$g$ of $X$ is positive. If
$k$ is archimedean we let $\mu_\xx$ be the Arakelov probability
measure on
$X(\overline{k})$. One way of giving $\mu_\xx$ is as follows: let $\iota \colon
X(\overline{k}) \to J(\overline{k})$ be an immersion of $X(\overline{k})$ into the
complex torus $J(\overline{k})$, where $J= \mathrm{Pic}^0 X$ is the jacobian
of $X$. Then $\mu_\xx = \frac{1}{g} \iota^* \nu$, where $\nu$ is the unique
translation-invariant
$(1,1)$-form representing the first Chern class of the line bundle
defining the canonical principal polarisation on $J(\overline{k})$.

Now suppose that $k$ is non-archimedean. Let $\mathcal{R}$ be the reduction
graph in the sense of Chinburg-Rumely \cite{cr} of $X$. This is a metrised graph, receiving a canonical surjective continuous specialisation map $\mathrm{sp} \colon |\xx| \to
\mathcal{R}$. In particular $\mathcal{R}$ is compact and path-connected. The map
$\mathrm{sp}$ has a canonical section $i \colon \mathcal{R} \to |\xx|$, and via $i$ the reduction graph $\mathcal{R}$ is identified with the minimal skeleton of $\xx$. The Laplace operator on $\xx$ extends in a natural way the Laplace operator on $\mathcal{R}$.

Now in \cite{zh}, Section~3 a canonical probability measure $\mu_\mathcal{R}$ is constructed on $\mathcal{R}$, called the admissible measure. We will not recall its definition; let us just say that it satisfies properties analogous
to the Arakelov measure in the archimedean setting. For example, it gives rise to
an adjunction formula. We define the canonical Arakelov measure $\mu_\xx$ on $|\xx|$ by putting $\mu_\xx = i_* \mu_\mathcal{R}$.

Let $g_{\mu_\xx,p}$ the Green's function based on $\mu_\xx$ from Proposition         \ref{properties}(iii). We then obtain a canonical symmetric pairing
$(,)_a$ on $X(\overline{k})$ by putting $(p,q)_a = g_{\mu_\xx,p}(q)$ for $p \neq q$.
This pairing coincides with the admissible pairing constructed in \cite{zh}, Section~4 using Green's functions on $\mathcal{R}$ with respect to $\mu_\mathcal{R}$ and the specialisation map. We refer to \cite{he} where this connection is made explicit.

We have the following proposition relating the integrals $\int_\xx \log |f| \, \mu_\xx$ to Zhang's admissible pairing.
\begin{prop} \label{formulaintegral}
Let $f$ be a non-zero rational function on $X \otimes \overline{k}$. Assume that a coordinate has been chosen on $\pp^1_k$. Then the formula:
\[ \int_\xx \log |f| \, \mu_\xx = \log |f|(r) + (\mathrm{div} f, r)_a \]
holds. Here $r$ is an arbitrary point in $X(\overline{k})$ outside
the support of $f$.
\end{prop}
\begin{proof} By Proposition \ref{properties}(i) we have:
\[ \ddc \log |f| = \delta_{\mathrm{div} f} \, . \]
By integrating against $g_{\mu_\xx,r}$ we obtain:
\[ \int_\xx g_{\mu_\xx,r} \ddc \log |f| = g_{\mu_\xx,r}(\mathrm{div}
f)=(\mathrm{div} f, r)_a \, . \]
On the other hand, by Proposition \ref{properties}(ii) and (iii) we have:
\begin{align*}
\int_\xx g_{\mu_\xx,r} \ddc \log |f| & = \int_\xx (\log |f|) \ddc g_{\mu_\xx,r}
\\
  & = \int_\xx (\log |f|) ( \mu_\xx - \delta_r) \\
  & = -\log|f|(r)+\int_\xx \log|f| \, \mu_\xx  \, .
\end{align*}
The proposition follows.
\end{proof}
In \cite{zh}, Theorem 4.6(iii) it is stated that with notations as in the above proposition $\log|f|(r)+(\mathrm{div} f,r)_a$ is constant outside the support of $f$. Using Thuillier's thesis we are thus able to interpret this constant as a
suitable integral over $\xx$.

\section{Connection with the N\'eron-Tate height on the jacobian} \label{global}

We now apply the global admissible intersection theory as developed in \cite{zh}, Section~5. Let $K$ be a number field and let $X$ be a smooth projective curve of positive genus $g$ over $K$.  For each place $v$ of $K$ we denote by $\overline{K}_v$ an algebraic closure of $K_v$. We endow each $K_v$ with a (standard) absolute value
$|\cdot|_v$ as follows. If $v$ is archimedean, we take the euclidean norm on
$K_v$. If $v$ is non-archimedean, we choose $|\cdot|_v$ such that
$|\pi|_v=\mathrm{e}$, where $\pi$ is a uniformiser of $K_v$.
Let $X_v$ be the
analytic space associated to $X \otimes \hat{\overline{K}}_v$,
and $\mu_v$ be the canonical measure on $X_v$, as in Section~\ref{potential}.

Let $M_K$ be the set of places of $K$. For each place $v \in M_K$, let $n_v$ be the (standard) local factor defined as follows: if $v$ is real, then $n_v=1$; if $v$ is complex, then $n_v=2$; if $v$ is non-archimedean, then $n_v$ is the $\log$ of the cardinality of the residue field at $v$. Note that we have a product formula
$\sum_{v \in M_K} n_v \log |\alpha|_v=0$ for all $\alpha$ in $K^*$.

Let $o \in X(K)$ be a point and $J = \mathrm{Pic}^0 \, X$ be the jacobian of $X$. We denote by $h \colon J(\overline{K}) \to \rr$ the N\'eron-Tate height associated to the canonical principal polarisation of $J$. Let $f$ be a non-zero rational function on $X$ and assume that a coordinate has been chosen on $\pp^1_K$. It follows from Proposition \ref{formulaintegral} that the real number $\int_{X_v} \log |f|_v \, \mu_v$ vanishes for almost all $v$.
\begin{prop} \label{NTheightformula}
Let $\mathrm{div} \, f = \sum_q m_q \cdot q$ be the divisor of the non-zero rational function $f$ on $X \otimes \overline{K}$. Then the formula:
\[ [K \colon \qq] \sum_q m_q \, h([q-o]) = g \sum_{v \in M_K} n_v \int_{X_v} \log |f|_v \, \mu_v \]
holds.
\end{prop}
\begin{proof} We follow the formalism and results of \cite{zh}, Section~5. First of all, for any point $q \in X(K)$ one has that $-2[K \colon \qq]h([q-o]) = (q-o,q-o)_a$, where now $(,)_a$ denotes global admissible pairing. This formula is essentially due to Hriljac and Faltings. By the adjunction formula (see \emph{op. cit.}) we next have $(q-o,q-o)_a = -(\omega+2\, o,q)_a-(\omega,o)_a$, where $\omega$ denotes the admissible relative dualising sheaf of $X$. We obtain:
\begin{align*} -2 \, [K\colon \qq]\sum_q m_q \, h([q-o]) & = -(\omega+2 \, o, \mathrm{div}\, f)_a \\
  & = -(\omega+2 \, o, \sum_{v \in M_K} n_v (\int_{X_v} \log|f|_v \, \mu_v ) \cdot X_v)_a \\
  & = -2g \sum_{v \in M_K} n_v \int_{X_v} \log |f|_v \, \mu_v \, .
\end{align*}
The proposition follows.
\end{proof}

\section{The local canonical height} \label{lambda}

Now assume that $X$ is equipped with a Galois covering map $x \colon X \to \pp^1_K$ of degree $N$, and that it has a point $o \in X(K)$ for which $x$ is totally ramified. Also assume that $x(o)=\infty$. Let again $v$ be a place of $K$. Take a point $p \in X(\overline{K}_v) \setminus \{o\}$ and put $f=x-x(p)$. This is then a well-defined rational function on $X \otimes \overline{K}_v$. We define $\lambda_v(p)$ to be the integral:
\[ \lambda_v(p) = \frac{1}{N} \int_{X_v} \log |x-x(p)|_v \, \mu_v \, . \]
This will be our basic object of study from now on. Let $p \in X(K) \setminus \{o\}$ and as in Section~\ref{global} let $h \colon J(\overline{K}) \to \rr$ be the N\'eron-Tate height associated to the canonical principal polarisation of the jacobian $J$ of $X$.
\begin{prop} \label{NT}
The real number $\lambda_v(p)$ vanishes for almost all $v$, and the formula:
\[ [K \colon \qq] h([p-o]) = g \sum_{v \in M_K} n_v \, \lambda_v(p) = \frac{g}{N} \sum_{v \in M_K} n_v \int_{X_v} \log |x-x(p)|_v \, \mu_v \]
holds.
\end{prop}
\begin{proof} This is a straightforward consequence of Proposition \ref{NTheightformula}.
Let $ -N o + \sum_q m_q q$ be the divisor of $f=x-x(p)$. The formula in Proposition \ref{NTheightformula} specializes to the following:
\[ [K \colon \qq] \sum_q m_q \, h([q-o]) = gN \sum_{v \in M_K} n_v \, \lambda_v(p) \, . \]
As all the $h([q-o])$ where $q$ runs through $x^{-1}(x(p))$ are equal (use the automorphism group of $x \colon X \to \pp^1_K$) we obtain the formula of the proposition.
\end{proof}
We find that the $\lambda_v(p)$ sum up to essentially the N\'eron-Tate height of the Abel-Jacobi image of $p$, with reference point $o$. It is instructive to compare this result with the main result of \cite{pst}, which writes the global canonical height associated to a morphism (dynamical system) $\varphi \colon \pp^1_K \to \pp^1_K$ as a sum of local canonical heights given by logarithmic integrals \`a la $\lambda_v$.

For $(X,o)$ an elliptic curve, the function $\lambda_v$ coincides with Tate's classical N\'eron function (see for instance \cite{se}, Chapter 6.5, the `first normalization').
\begin{prop} Assume that $(X,o)$ is an elliptic curve and that $x \colon X \to \pp^1_K$ is a hyperelliptic map. Let $v$ be a place of $K$. Then $\lambda_v$ is equal to the unique N\'eron function with respect to $o$ on $(X,o)$ normalised such that $\lambda_v(p) - \frac{1}{2} \log|x(p)|_v \to 0$ as $p \to o$.
\end{prop}
\begin{proof} Let $\tilde{\lambda}_v \colon X(\overline{K}_v) \setminus \{o\} \to \rr$ be the N\'eron function described in the proposition. It satisfies the `quasi-parallellogram law' (cf. \emph{op. cit.}):
\[
\tilde{\lambda}_v(p+q) + \tilde{\lambda}_v(p-q)=
2\, \tilde{\lambda}_v(p) + 2 \, \tilde{\lambda}_v(q)-\log|x(p)-x(q)|_v
\]
for all $p,q \in X(\overline{K}_v) \setminus \{o \}$ such that $p \neq \pm q$.
By fixing $p$ and integrating against $\mu_v(q)$ one finds, using the translation-invariance of $\mu_v$ and cancelling three terms:
\[ \tilde{\lambda}_v(p) = \frac{1}{2} \int_{X_v} \log|x-x(p)|_v \, \mu_v \, . \]
This shows that $\tilde{\lambda}_v = \lambda_v$.
\end{proof}
The following theorem analyzes the properties of the local functions $\lambda_v$ in more detail. Let $v$ again be any place of $K$.
\begin{thm} \label{equationslambda}
The function $\lambda_v \colon X(\overline{K}_v) - \{o\} \to \mathbb{R}$
extends naturally and uniquely as a $(0,0)$-current on $X_v$.
It satisfies the $dd^c$-equation:
\[ \ddc \lambda_v = \mu_v - \delta_o \, . \]
As a consequence we have:
\[ \lambda_v(p) = (p,o)_a + \int_{X_v} \lambda_v \, \mu_v \, , \]
where $(,)_a$ is the local admissible pairing on $X(\overline{K}_v)$.
Furthermore we have:
\[ \lambda_v(p) - \frac{1}{N} \log |x(p)|_v \to 0  \]
as $p \to o$ on $X(\overline{K}_v)$. In particular $\lambda_v$ defines a local
Weil function with respect to the divisor $o$ on $X$.
\end{thm}
\begin{proof}
Let $G$ be the automorphism group
of $x \colon X \to \pp^1_K$ over $\overline{K}$. Note that $\mathrm{div}(x-x(p))=-No + \sum_{\sigma \in G}
\sigma(p)$. From Proposition \ref{formulaintegral} we obtain that:
\begin{equation} \label{keyequation}
N\lambda_v(p) = \log|x(p)-x(r)|_v + \sum_{\sigma \in G} \left( \sigma(p)-o,r
\right)_a  \, ,
\end{equation}
where $r \in X(\overline{K}_v)$
is an arbitrary point, not in
the support of $x-x(p)$. Now consider equation (\ref{keyequation}) with $p$ as a variable and $r$ fixed.
Both $(\sigma(p)-o,r)_a$ and $\log|x(p)-x(r)|_v$ extend as $(0,0)$-currents over
$X_v$. Hence so does $\lambda_v$. The extension is unique, as
$X(\overline{K}_v)$ is
dense in $X_v$.  To prove the
first formula, note that as $(,)_a$ is canonical, it is invariant
under $G$. We can thus rewrite
(\ref{keyequation}) as:
\[  N \lambda_v(p) = \log|x(p)-x(r)|_v + \sum_{\sigma \in G}
\left(p-o,\sigma(r)\right)_a   \, . \]
Taking $dd^c$ we have, by Proposition \ref{properties}:
\[ N \ddc \lambda_v= \sum_{\sigma \in G} \left( \delta_{\sigma(r)} -  \delta_o
\right) + \sum_{\sigma \in G} \left( \mu_v - \delta_{\sigma(r)} \right)  \, . \]
It follows that $\ddc \lambda_v = \mu_v - \delta_o$ as required.
As $(p,o)_a = g_{\mu_v,o}(p)$ satisfies the same $\ddc$-equation, we obtain the second formula. To prove the last formula, let $p \to o$ in
(\ref{keyequation}). Then the sum $\sum_{\sigma \in G} \left( \sigma(p)-o,r
\right)_a$ converges to zero.
\end{proof}
\begin{thm} Let $(o,o)_a$ be the admissible self-intersection of the point $o$ on $X$. Then the formula:
\[ \sum_{v \in M_K} n_v \int_{X_v} \lambda_v \, \mu_v = \frac{1}{N} \sum_{v \in M_K} n_v \int_{X_v} \int_{X_v} \log |x(p)-x(q)|_v \mu_v(p) \mu_v(q) = -(o,o)_a \]
holds.
\end{thm}
\begin{proof} Choose a $p \in X(K) \setminus \{o\}$ arbitrarily. Let again $G$ be the automorphism group of $x$ over $\overline{K}$.
Note that $\mathcal{O}_X(\mathrm{div}(x-x(p)))$ = $\mathcal{O}_X(-No + \sum_{\sigma \in G} \sigma(p))$ is a trivial admissible line bundle at all places $v$ of $K$. It follows that the global pairing $(-No + \sum_{\sigma \in G} \sigma(p),r)_a$ is independent of the choice of $r$. We can choose $r=o$ and we derive from equation (\ref{keyequation}) in global admissible theory the relation:
\[ N \sum_{v \in M_K} n_v \, \lambda_v(p) = (-No + \sum_{\sigma \in G} \sigma(p),o)_a = N(p-o,o)_a \]
and hence:
\[ \sum_{v \in M_K} n_v \, \lambda_v(p) = (p-o,o)_a = (p,o)_a - (o,o)_a \, .\]
By Theorem \ref{equationslambda} we have on the other hand:
\[ (p,o)_a = \sum_{v \in M_K} n_v \left(\lambda_v(p) - \int_{X_v} \lambda_v \, \mu_v \right) \, . \]
The theorem follows.
\end{proof}
Note that if $g \geq 1$ and $(2g-2)o$ is a canonical divisor on $X$, we have:
\[ ((2g-2)o-\omega,(2g-2)o-\omega)_a=0  \, ,  \quad \textrm{i.e.,} \quad
-(o,o)_a = \frac{ (\omega,\omega)_a}{4g(g-1)} \, .
\]
In that case the formula in the Theorem becomes:
\[ \sum_{v \in M_K} n_v \int_{X_v} \lambda_v \, \mu_v = \frac{1}{N} \sum_{v \in M_K} n_v \int_{X_v} \int_{X_v} \log |x(p)-x(q)|_v \mu_v(p) \mu_v(q) = \frac{ (\omega,\omega)_a}{4g(g-1)} \, . \]
The condition that $(2g-2)o$ is a canonical divisor is fulfilled when $X$ is a hyperelliptic curve, and more generally, when the coordinate ring of $X \setminus \{o\}$ is generated by two elements. Such curves go by different names in the literature: plane model curves, $C_{ab}$-curves, Burchnall-Chaundy curves, $\ldots$

\section{Hyperelliptic curves} \label{hyperelliptic}

The purpose of this section is to study the local invariants $\int_{X_v} \lambda_v \, \mu_v$ in more detail for hyperelliptic maps. In particular we connect them with the local invariants $\chi(X_v)$ introduced in \cite{dj}. We start with a rather remarkable formula that computes the special value of $\lambda_v$ at a hyperelliptic ramification point.

Let $(X,o)$ be a hyperelliptic curve of genus $g \geq 2$ over $K$ given by
an equation $y^2=f(x)$, where $f(x) \in K[x]$ is monic and separable of
degree $m=2g+1$. Fix a place $v$ of $K$, as well as an algebraic closure
$\overline{K}_v$ of $K_v$. Keep the $v$-adic absolute values on $K_v$ and
$\overline{K}_v$ as defined in Section \ref{global}.
\begin{prop} \label{atalpha}
Let $w \in X(\overline{K}_v) \setminus \{o\}$ be a hyperelliptic ramification
point of $X$ and let $\alpha = x(w)$ in $\overline{K}_v$.
Then the formula:
\[ 2 \, \lambda_v(w) = \int_{X_v} \log|x-\alpha|_v \, \mu_v =
\frac{1}{2g} \log |f'(\alpha)|_v \]
holds.
\end{prop}
\begin{proof} We use a result on the arithmetic of symmetric roots from
\cite{dj}. Let $\alpha_1,\ldots,\alpha_{2g+2}=\infty$ on $\mathbb{P}^1(\overline{K}_v)$ be
the branch points of $x$. The symmetric root of a triple
$(\alpha_i,\alpha_j,\alpha_k)$ of distinct branch points is then defined to be
an element:
\[ \ell_{ijk} = \frac{\alpha_i - \alpha_k}{\alpha_j - \alpha_k} \sqrt[2g]{ -
\frac{f'(\alpha_j)}{f'(\alpha_i)}} \]
of $\overline{K}_v^*$. The actual choice of $2g$-th root will be immaterial in
the discussion below. If $\alpha_j$ equals infinity, the formula is to be read as follows:
\begin{equation} \label{symmrootinfinity}
\ell_{i \infty k} = \left( \alpha_i - \alpha_k \right) \sqrt[2g]{-f'(\alpha_i)}^{-1}
\end{equation}
(recall that $f$ is monic). Now let $w_1,\ldots,w_{2g+2}$ on $X(\overline{K}_v)$ be the
hyperelliptic ramification points corresponding to $\alpha_1, \ldots,
\alpha_{2g+2}$. Theorem C of \cite{dj} then states that if $(w_i,w_j,w_k)$ is a
triple of distinct ramification points, the formula:
\begin{equation} \label{rootandpairing}
(w_i-w_j,w_k)_a = -\frac{1}{2} \log |\ell_{ijk}|_v
\end{equation}
holds. Here, as before $(,)_a$ denotes Zhang's local
admissible pairing on $X(\overline{K}_v)$. Applying Proposition \ref{formulaintegral}
to the rational function $x-\alpha_i$, with $\alpha_i$ a finite branch point, we
find:
\[ \int_{X_v} \log|x-\alpha_i|_v \, \mu_v =  \log|x(p)-\alpha_i|_v + 2 \, (w_i-o,p)_a \]
for any $p \neq o,w_i$. Taking $p=w_k$ and applying
(\ref{rootandpairing}) we find:
\begin{align*}
\int_{X_v} \log|x-\alpha_i|_v \, \mu_v & =  \log|\alpha_i-\alpha_k|_v + 2 \, (w_i-o,w_k)_a \\
  & = \log|\alpha_i - \alpha_k|_v - \log |\ell_{i \infty k}|_v \, .
\end{align*}
Hence by (\ref{symmrootinfinity}):
\[ 2 \, \lambda_v(w_i)=\int_{X_v} \log|x-\alpha_i|_v \, \mu_v=\frac{1}{2g}
\log|f'(\alpha_i)|_v \, . \]
The proposition is proven.
\end{proof}
The function $\lambda_v$ depends on the choice of monic equation $f$ for the
pointed curve $(X,o)$. Let $\Delta =
2^{4g}\Delta(f)$ where $\Delta(f)$ is the discriminant of $f$. We refer to
\cite{lo} for properties of $\Delta$. The discriminant $\Delta$
generalises the usual definition $\Delta = 2^4 \Delta(f)$ in the case where
$(X,o)$ is an elliptic curve.
We renormalise $\lambda_v$ by putting:
\[ \hat{\lambda}_v(p) = \lambda_v(p) - \frac{1}{4g(2g+1)} \log |\Delta|_v \, . \]
Then $\hat{\lambda}_v$ is independent of the choice of monic equation
$f$ for $(X,o)$, as one checks by replacing $x$ by $u^2x+t$ for $u \in K^*$, $t
\in K$. We obtain the familiar renormalisation:
\[ \hat{\lambda}_v = \lambda_v - \frac{1}{12} \log |\Delta|_v \]
in the case where $(X,o)$ is an elliptic curve (cf. \cite{se}, Chapter 6.5, the `second normalization').

Let $i$ be an index with $1 \leq i \leq 2g+2$. Then put:
\[ \chi(X_v) = -2g\left( \log|2|_v + \sum_{k \neq i} (w_i,w_k)_a \right) \, . \]
It is proved in \cite{dj}, Theorem B that $\chi(X_v)$ is independent of the choice
of~$i$, hence is an invariant of $X_v$. One can prove that $\chi(X_v) \geq 0$ for all $v$, and in fact $\chi(X_v)>0$ if $v$ is archimedean. We have $\chi(X_v)=0$ if $v$ is non-archimedean and $X$ has potentially good reduction at $v$. Let $(\omega,\omega)_a$ be the admissible self-intersection of the relative dualising sheaf of $X$. We have:
\begin{equation} \label{omegasquared}
(\omega,\omega)_a = \frac{2g-2}{2g+1} \sum_{v \in M_K} n_v \, \chi(X_v) \, ,
\end{equation}
where $v$ runs over the places of $K$. The following result says that
$\chi(X_v)$ is essentially equal to $\int_{X_v} \hat{\lambda}_v \, \mu_v$.
\begin{cor} \label{integrallambdahat} Let $v$ be a place of $K$.
Then the formula:
\[ \chi(X_v) = 2g(2g+1) \int_{X_v} \hat{\lambda}_v \, \mu_v  \]
holds.
\end{cor}
\begin{proof} Let $w_1,\ldots,w_{2g+1}$ be the ramification points of $X$ distinct from $o$. From Proposition \ref{atalpha} and the identity $\Delta(f)=\prod_{i=1}^{2g+1} f'(\alpha_i)$ we obtain the formula:
\[ \sum_{i=1}^{2g+1} \hat{\lambda}_v(w_i) = -\log|2|_v \, . \]
From Proposition \ref{equationslambda} we obtain that:
\[ \hat{\lambda}_v(p) = (p,o)_a + \int_{X_v} \hat{\lambda}_v \, \mu_v  \]
for each $p \in X(\overline{K}_v) \setminus \{o\}$. The corollary follows by combining these two relations.
\end{proof}
Note that the above result yields an alternative approach to Theorem B. Namely we derive:
\[ \sum_{v \in M_K} n_v \int_{X_v} \hat{\lambda}_v \, \mu_v = \frac{ (\omega,\omega)_a}{4g(g-1)} \]
from (\ref{omegasquared}) and Corollary \ref{integrallambdahat},
with the additional information that in this local decomposition, all contributions on the left hand side are non-negative, and are canonically associated to $X$. The equivalence of the above local decomposition with the formula from Theorem B is clear by the product formula.

It would be interesting to have explicit formulas for $\hat{\lambda}_v(p)$, for genus two curves say, \`a la the ones of Tate (see \cite{sil2}, Chapter VI) in the context of elliptic curves, given the type of the reduction graph $\mathcal{R}_v$ of $X$ at $v$, and the specialisation of $p$ on $\mathcal{R}_v$, if $v$ is non-archimedean. A natural case to start would be the case where $X$ is a Mumford curve at $v$. This occurs if the branch points of $X$ come in pairs of points closer to one another than to the other branch points, where the distance is measured by rational affinoid subsets of
the projective line \cite{br}.

\section{Proof of Theorem C} \label{theoC}

In this section we prove Theorem C. Let $K$ be a number field. We will make use of the following general diophantine approximation result due to Faltings (see \cite{fa}, Theorem II):
\begin{thm} \label{faltings}
Let $A$ be an abelian variety over $K$ and let $D$ be an ample
divisor on $A$. Let $v$ be a place of $K$ and let $\lambda_{D,v}$ be a N\'eron
function on $A(\overline{K}_v)$ with respect to $D$. Let $h$ be a height on $A(\overline{K})$
associated to an ample line bundle on $A$, and let $\kappa \in \rr_{>0}$. Then
there exist only finitely many $K$-rational points $z$ in $A-D$ such that
$\lambda_{D,v}(z) > \kappa \cdot h(z)$.
\end{thm}
Again let $X$ be a smooth projective curve of positive genus $g$ defined over $K$, let $x \colon X \to \pp^1_K$ be a Galois covering map and assume that $o$ is a totally ramified point for $x$, such that $x(o)=\infty$. Let $J=\mathrm{Pic}^0 X$ be the jacobian of $X$, and let
$\iota \colon X \to J$ be the Abel-Jacobi embedding given by $p \mapsto [p-o]$.
Then we have a natural theta divisor $\Theta$ on $J$ represented by the classes
$[q_1+\cdots+q_{g-1} - (g-1)o]$ for $q_1,\ldots,q_{g-1}$ running through $X$. We note that $\Theta$ is invariant under the automorphism group of $x \colon X \to \pp^1_K$ acting on $J$ in its natural fashion.

For $n \geq g$ any integer we define $X_n$ to be the divisor
$\iota^* [n]^* \Theta$ of $X$. This is an effective $K$-divisor on $X$ of degree
$gn^2$, as can be seen for example by noting that $X_n$ coincides with the
divisor of Weierstrass points of the line bundle $\mathcal{O}_X(o)^{\otimes
n+g-1}$ as considered in \cite{ne}. The points in the support of $X_n$ are called the $n$-th order Weierstrass points of $(X,o)$. By our remark above, $X_n$ is invariant under the automorphism group of $x \colon X \to \pp^1_K$. We note that $X_g = W + go$, where $W=\sum_{p \in X} w(p) \cdot p$ is the `classical' divisor of Weierstrass points of $X$ determined by the Weierstrass weights $w(p)$ based on the gap sequence at $p$.

For $p \in X(\overline{K})$ we define: 
\[ T(p) = \{ n \in \zz_{\geq g} \,|\, p \notin X_n \} = \{ n \in \zz_{\geq g} \,|\, n[p-o] \notin \Theta \} \, . \]
For example, if $g=1$ then $\Theta = \{o\}$ and $T(p)$ is the set of positive integers in the complement of a subgroup of $\zz$.
\begin{lem} \label{infiniteset}
Let $p \in X(\overline{K})$. If $T(p)$ is not empty, then $T(p)$ contains
infinitely many elements.
\end{lem}
\begin{proof} For $p$ such that $[p-o]$ is torsion in $J$ the statement is
immediate: assume $n_0[p-o] \notin \Theta$, then if $k$ is the order of $[p-o]$
we can take those $n \geq g$ such that $n \equiv n_0 \bmod k$. Assume therefore
that $[p-o]$ is not torsion in $J$. We prove that infinitely many points of the
form $n[p-o]$ where $n \in \zz_{\geq 0}$ are not in $\Theta$. Let $Z^+$ be the Zariski closure of the set $\zz_{\geq 0} \cdot [p-o]$ in $J$, and
let $Z$ be the Zariski closure in $J$ of the
subgroup $ \zz \cdot [p-o]$ of $J$. Then $Z$ is a closed algebraic
subgroup of $J$, by Lemma \ref{isagroup} below, and using the involution $x \mapsto -x$ on $J$ one sees that actually $Z^+=Z$. Suppose that only finitely many of the $n[p-o]$ with $n \in \zz_{\geq 0}$ are outside
$\Theta$. Then $Z=Z^+$ is the union of a finite positive number of isolated points and a closed subset of $\Theta$. It follows that $Z$ has dimension zero, contradicting
the assumption that $[p-o]$ is not torsion.
\end{proof}
\begin{lem} \label{isagroup} Let $G$ be an algebraic group variety over a
field $k$ and let $H$ be a subgroup of $G$. Then the
Zariski closure of $H$ in $G$ is an algebraic subgroup of $G$.
\end{lem}
\begin{proof} Let $Z$ be the Zariski closure of $H$ in $G$ and for every $h$ in
$H$ denote by $t_hZ$ the left translate of $Z$ under $h$ in $G$.
As $t_hZ$ is closed in $G$ and contains $H$ we find that
$t_hZ$ contains $Z$ and in fact
$t_hZ = Z$. This implies that $H$ is contained in the stabiliser
$\mathrm{Stab}(Z)$ of $Z$, which is a closed algebraic subgroup of $G$. We
conclude that $Z$ is contained in $\mathrm{Stab}(Z)$ and hence
$Z$ is itself an algebraic subgroup of $G$.
\end{proof}
Note that $T(p)$ can be empty for $p \neq o$: for example if $x \colon X \to
\mathbb{P}^1_K$ is totally ramified at $p$ as well and $N \leq g$. We refer to \cite{sv} for a study of the sets $T(p)$ in a more general setting.

We have the following theorem.
\begin{thm} \label{main}
Let $p \in X(K) \setminus \{o\}$ be a rational point.
Assume that $p$ is not in $X_n$ for at least one $n \geq g$. Let $v$ be a place of~$K$. Then one has:
\[ \frac{1}{gn^2} \sum_{q \in X_n \atop x(q) \neq \infty} \log|x(p)-x(q)|_v
\longrightarrow \int_{X_v} \log|x - x(p)|_v \, \mu_v \]
as $n \to \infty$ over the infinite set $T(p)$. In the sum on the left hand side, points are counted with multiplicity.
\end{thm}
We note that Theorem \ref{main} implies Theorem C. Indeed, if $n \in T(p)$ then
$x(q) \neq x(p)$ for all $q \in X_n$ since $X_n$ is invariant under the
automorphism group of $x \colon X \to \pp^1_K$ over $\overline{K}$, and this
automorphism group acts transitively on each fiber of $x$.

The proof of Theorem \ref{main} is based on the existence of an identity:
\[
 \log|a(n)|_v + \frac{1}{N}
\sum_{q \in X_n \atop x(q) \neq \infty} \log|x(p)-x(q)|_v =
-\lambda_{\Theta,v}(n[p-o]) + gn^2 \lambda_v(p)
\]
of (generalized) functions on $X(\overline{K}_v)$  where $\lambda_{\Theta,v}$ is a N\'eron function with respect to $\Theta$ on $J(\overline{K}_v)$ and where $a(n)$ is a function
with polynomial growth in $n$. We obtain Theorem \ref{main} by dividing by
$gn^2$ and letting $n$ tend to infinity over $T(p)$, using Faltings's result to
see that $ \lim_{n \to \infty} \lambda_{\Theta,v}(n[p-o])/gn^2 = 0 $. Let's make these ideas precise now.
\begin{prop} \label{identityprop} Let $v$ be a place of $K$, and let $\lambda_{\Theta,v}$ be a N\'eron function with respect to $\Theta$ on $J(\overline{K}_v)$.
There exists a polynomial $a(u) \in \overline{K}_v[u]$ such that
for all integers $n$ with $n \geq g$ and for all $p \in
X(\overline{K}_v)$ with $p \notin X_n$, the equality:
\[   \log|a(n)|_v + \frac{1}{N}
\sum_{q \in X_n \atop x(q) \neq \infty} \log|x(p)-x(q)|_v =
-\lambda_{\Theta,v}(n[p-o]) + gn^2 \lambda_v(p) \]
holds. In the sum, points are counted with their multiplicity.
\end{prop}
\begin{proof} Fix an integer $n \geq g$. First of all neglect the term $\log |a(n)|_v$ on the left hand side. Write $\ell_{n,v}(p)$ as a shorthand
for $\lambda_{\Theta,v}(n[p-o])$. One can view $L=\mathcal{O}_J(\Theta)$ as an
adelic line bundle on $J$ by putting $\|1\|_{L,v}(z) =
\exp(-\lambda_{\Theta,v}(z))$ where $1$ is the canonical global section of
$\mathcal{O}_J(\Theta)$. By pullback one obtains a structure of adelic
line bundle on each $L_n=\mathcal{O}_X(X_n)
=\iota^*[n]^*\mathcal{O}_J(\Theta)$ given by $\|1\|_{L_n,v}(p) =
\exp(-\ell_{n,v}(p))$ where now $1$ is the canonical global section of
$\mathcal{O}_X(X_n)$. By \cite{zh}, Section~4.7 (see also \cite{he}, Section~4) the resulting adelic metric is  admissible; in particular $\ell_{n,v}(p)$
is equal to the admissible pairing $(p,X_n)_a$ up to an additive constant. As a result
$\ell_{n,v}$ extends to $\mathcal{D}^0(X_v)$, the space of $(0,0)$-currents on
$X_v$. As the other terms
in the equality to be proven do so as well, we try to prove the equality as an
identity in $\mathcal{D}^0(X_v)$. By Proposition \ref{properties}(iii) we are
done once we prove that both sides of the claimed equality have the same image
under $dd^c$, and the difference of both sides tends to zero as $p$ tends to
$o$ over $X(\overline{K}_v)$, avoiding $X_n$. From the observation that  $\|1\|_{L_n,v}(p) =
\exp(-\ell_{n,v}(p))$ defines an admissible metric on $L_n$ we obtain first of
all that:
\[ \ddc \ell_{n,v} = (\deg X_n) \mu_v - \delta_{X_n} = gn^2 \mu_v - \delta_{X_n}
\, .    \]
We obtain from Proposition \ref{equationslambda} that: 
\[ \ddc \lambda_v = \mu_v - \delta_o \, . \]
Finally by the Poincar\'e-Lelong equation Proposition \ref{properties}(i) we have:
\[ \ddc \frac{1}{N}
\sum_{q \in X_n \atop x(q) \neq \infty} \log|x(p)-x(q)|_v = \delta_{X_n} - gn^2
\delta_o \, ,\]
and hence the difference of both sides in the identity vanishes under $dd^c$. Now we consider the behavior of both left and right hand side as $p \to o$ avoiding $X_n$. Let $z_1,\ldots,z_g$ be local coordinates around the origin on $J(\overline{K}_v)$, and let $s \in \overline{K}_v[[z_1,\ldots,z_g]]$ be a local equation for $\Theta$ such that:
\[ \lambda_{\Theta,v}(z) + \log |s(z_1,\ldots,z_g)|_v \to 0\]
as $z \to 0$. Let  $t$ be a local equation for $o$ on $X(\overline{K}_v)$ and write $\iota^* z_j = a_j t^{m_j}(1 + O(t))$ with $a_j \in \overline{K}_v$, $m_j \in \zz_{>0}$. We have $[n]^*z_j \equiv n z_j \bmod (z_1,\ldots,z_g)^2$ and it follows that there exists a polynomial $a(u) \in \overline{K}_v[u]$ and $m \in \zz_{>0}$ such that:
\[ \iota^* [n]^* s(z_1,\ldots,z_g) = a(n) t^m(1+O(t)) \]
for all integers $n \geq g$. Note that $m$ is the multiplicity of $o$ in $X_n$; it is independent of $n$. We can assume without loss of generality that $t^{-N} =x$. We find:
\begin{align*}
\lambda_v(p) \longrightarrow & -\log|t(p)|_v \, , \\
\lambda_{\Theta,v}(n[p-o]) + \log|a(n)|_v  \longrightarrow & -m \log|t(p)|_v \, , \\
\frac{1}{N}
\sum_{q \in X_n \atop x(q) \neq \infty} \log|x(p)-x(q)|_v  \longrightarrow & -(gn^2-m)\log |t(p)|_v
\end{align*}
as $p \to o$ avoiding $X_n$. The proposition follows by combining these asymptotics.
\end{proof}
\begin{remark} Let $(X,o)$ be an elliptic curve. We can choose $\lambda_{\Theta,v}$ to be $\lambda_v$ itself; then $a(n)=n$ and the left hand side of the identity in Proposition \ref{identityprop} is equal to $\log |\psi_n(x(p))|_v$ with $\psi_n \in K[x]$ the $n$-th division polynomial of $(X,o)$ with respect to $x$ (cf. \cite{sil1}, Exercise 3.7). One finds the identity:
\[ \log |\psi_n(x(p))|_v = -\lambda_v(np) + n^2 \lambda_v(p) \]
which seems to be well known.
\end{remark}
\begin{remark} If $(X,o)$ is a hyperelliptic curve and $v$ is archimedean one finds in \cite{on}, Theorem 8.3 a holomorphic analogue of the identity in Proposition \ref{identityprop}, based on Klein's hyperelliptic sigma-function. It would be interesting to generalize the result from \cite{on} to the case of more general $(X,o)$.
\end{remark}
\begin{proof}[Proof of Theorem \ref{main}]
Let $p \in X(K) \setminus \{o\}$ be a point such that $T(p)$ is infinite, and let $v$ be a place of $K$. By Proposition \ref{identityprop} we are done once we prove
that $\log|a(n)|_v/n^2 \to 0$ as $n \to \infty$ and
$\lambda_{\Theta,v}(n[p-o])/n^2 \to 0$ as $n\to \infty$ over $T(p)$. The first
statement is immediate since $a(n)$ is a polynomial in $n$. As to the second
statement, note that it follows immediately if $[p-o]$ is torsion since then the
set of values $\lambda_{\Theta,v}(n[p-o])$ as $n$ ranges over $T(p)$ is bounded.
Assume therefore that $[p-o]$ is not torsion. Then the $n[p-o]$ with $n$ running
through $T(p)$ form an infinite set of $K$-rational points of $J-\Theta$.
Since:
\[ \frac{\lambda_{\Theta,v}(n[p-o])}{n^2} = h([p-o]) \cdot
\frac{\lambda_{\Theta,v}(n[p-o])}{h(n[p-o])} \]
with $h([p-o])>0$ Faltings's Theorem \ref{faltings} can be applied to
give:
\[ \limsup_{n \to \infty \atop n \in T(p)}
\frac{\lambda_{\Theta,v}(n[p-o])}{n^2} \leq 0 \, . \]
On the other hand $\lambda_{\Theta,v}$ is bounded from below so that:
\[ \liminf_{n \to \infty \atop n \in T(p)}
\frac{\lambda_{\Theta,v}(n[p-o])}{n^2} \geq 0 \, . \]
Theorem \ref{main} follows.
\end{proof}
We finish with an application of Theorem \ref{main}.
Let $(X,o)$ be a hyperelliptic curve over $K$ of genus $g \geq 2$ and let $y^2 =f(x)$ with $f \in O_K[x]$ monic of degree $2g+1$ be an equation for $(X,o)$ putting $o$ at infinity. Here $O_K$ denotes the ring of integers of $K$. In \cite{ca} polynomials $\psi_n \in O_K[x]$ are constructed with leading coefficient polynomially growing in $n$ and with zero divisor on $X$ equal to $X_n - m \cdot o$ where $m$ is the multiplicity of $o$ in $X_n$ (actually $m = g(g+1)/2$). The sequence $(\psi_n)_n$ is determined by a non-linear recurrence relation.
\begin{thm} \label{EW}
Assume that $p=(x(p),y(p)) \in O_K^2$ is an $o$-integral point of $X$, with $y(p) \neq 0$. Let $S$ be a finite set of places of $K$ containing the places of bad reduction for $X$ as well as the infinite ones. Then the identity:
\[ [K\colon \qq] h([p-o]) = \lim_{n \to \infty \atop n \in T(p)} \frac{1}{gn^2} \log \prod_{v \in S} |\psi_n(x(p))|^{n_v}_v \]
holds.
\end{thm}
\begin{proof} Let $\hat{\lambda}_v$ be the renormalised version of $\lambda_v$ introduced in Section \ref{hyperelliptic}. By the product formula and Proposition \ref{NT} we can write:
\[ [K\colon \qq] h([p-o]) = \sum_{v \in M_K} n_v \, \hat{\lambda}_v(p) \, . \]
By Theorem \ref{equationslambda}, for the local admissible pairing $(p,o)_a$ we have:
\[ \hat{\lambda}_v(p) = (p,o)_a + \int_{X_v} \hat{\lambda}_v \, \mu_v \, . \]
By the results in Section \ref{hyperelliptic} we have $\int_{X_v} \hat{\lambda}_v \, \mu_v =0$ if $v$ is a place of good reduction. Further we have $(p,o)_a=0$ if $v$ is a place of good reduction since $p$ is $o$-integral by assumption. We obtain that:
\[ [K\colon \qq] h([p-o]) = \sum_{v \in S} n_v \, \hat{\lambda}_v(p) \, . \]
By assumption $p$ is not a Weierstrass point of $X$. Hence according to Theorem \ref{main} we have:
\[ \hat{\lambda}_v(p) = -\frac{1}{4g(2g+1)} \log |\Delta|_v+\lim_{n \to \infty \atop n \in T(p)} \frac{1}{gn^2} \log |\psi_n(x(p))|_v   \]
for any place $v$ of $K$. By combining these two formulas and interchanging the limit and the (finite) sum we obtain the theorem, upon noting that:
\[ \sum_{v \in S} n_v \log |\Delta|_v = 0 \]
by the product formula and the fact that $\log |\Delta|_v=0$ if $v \notin S$.
\end{proof}
Theorem \ref{EW} generalizes the main result of \cite{ewNY} which is the analogue in the case of elliptic curves. For points $p$ that are not necessarily $o$-integral one has to make sure that the set $S$ contains the primes of $K$ where $p$ reduces to $o$; then the same formula works.

The limit formula in Theorem \ref{EW} gives a method, in principle, to approximate to high precision N\'eron-Tate heights of points on $X$ without exhibiting a model for the jacobian of $X$. Here one uses the results of \cite{ca} to compute recursively the sequence of division polynomials $\psi_n$. Note that in order to have the formula in Theorem \ref{EW} work in practice one has to know in advance that the gaps in the sequence $T(p)$ are bounded independent of $p$. Here is an argument that indicates that a gap between two consecutive integers in $T(p)$ should not be larger than~$g$: a gap of length $g+1$ would give rise to an element in an intersection $\Theta \cap \Theta_{[p-o]}\cap \ldots \cap \Theta_{g[p-o]}$ of $g+1$ translates of $\Theta$. These translates are distinct if $p$ is not a Weierstrass point. The intersection should be empty for dimension reasons.

\section*{Acknowledgments}
The research done for this paper was supported by VENI grant 639.033.402
from the Netherlands Organisation for Scientific Research (NWO). Part of the research was done at the Max Planck Institute in Bonn, whose hospitality is greatly acknowledged.

\newpage

\noindent Address of the author: \\  \\
Robin de Jong \\
Mathematical Institute \\
University of Leiden \\
PO Box 9512 \\
2300 RA Leiden \\
The Netherlands \\
Email: \verb+rdejong@math.leidenuniv.nl+

\end{document}